\documentclass[11pt]{amsart}
\usepackage{amsmath}
\usepackage{amsfonts,amssymb,amsthm,amstext,amscd,boxedminipage}  

\usepackage{graphicx}
\usepackage {psfrag} 
\usepackage{color}
\usepackage{tikz}
\usepackage{subfig}
\usetikzlibrary{shapes,snakes}
\graphicspath{{./Figures/}}	
 
 \numberwithin{equation}{section}

\DeclareMathSymbol{\minus} {\mathord}{operators}{"2D} %

\theoremstyle{plain}
\newtheorem{theorem}{Theorem}[section]
\newtheorem{maintheorem}[theorem]{Main Theorem}
\newtheorem{lem}[theorem]{Lemma}
\newtheorem{cor}[theorem]{Corollary}

\theoremstyle{definition}
\newtheorem{df}[theorem]{Definition}
\newtheorem{remark}[theorem]{Remark}
\newtheorem{example}[theorem]{Example}

\def \R {\mathbb{R}}
\def \C {\mathbb{C}}
\def \Q {\mathbb{Q}}
\def \H {\mathbb{H}}
\def \Z {\mathbb{Z}}
\def \D {\mathbb{D}}
\def \M {\mathcal M}
\def \CN {\mathbb N}
\def \CO {\mathcal O}
\def \CR {\mathcal R}
\def \CS {\mathcal S}
\def \CW {\mathcal W}
\def \rk {{\mbox rank}}

\begin{document}

\title[Head and Tail for Adequate Knots]{The Head and Tail of the Colored Jones Polynomial for Adequate Knots}
\date{\today}

\author[C. Armond]{Cody Armond}
\address{Department of Mathematics, University of Iowa,
Iowa City, IA 52242-1419, USA}
\email{cody-armond@uiowa.edu}

\author[O. T. Dasbach]{Oliver T. Dasbach}
\address{Department of Mathematics, Louisiana State University,
Baton Rouge, LA 70803, USA}
\email{kasten@math.lsu.edu}
\thanks {The second author was supported in part by NSF grant DMS-1317942. The first author was partially supported as a graduate student by NSF VIGRE grant DMS 0739382.}

\begin{abstract}
We show that the head and tail functions of the colored Jones polynomial of adequate links are the product of head and tail functions of the colored Jones polynomial of alternating links that can be read-off an adequate diagram of the link. We apply this to strengthen a theorem of Kalfagianni, Futer and Purcell on the fiberedness of adequate links.
\end {abstract}
 
\maketitle

\def\kbsm#1{\mathscr{S}_K(#1)}
\def\sgn#1;#2{\mathbb{S}_{#1,#2}} 
\def\mcg#1;#2{\Gamma_{#1,#2}} 
\def\fg#1;#2{\Pi_{#1,#2}}
\def\tb#1;#2{\mathscr{K}_{\frac{#1}{#2}}}
\def\periph{(\mathcal{\mu},\mathcal{\lambda})}
\def\ext#1{\mathscr{E}(\mathscr{#1})}
\def \qP #1 #2 #3 {(#1;#2)_{#3}}
 
\newcommand{\E}{\mathcal{E}}\def \R {\mathbb{R}}
\def \C {\mathbb{C}}
\def \H {\mathbb{H}}
\def \Z {\mathbb{Z}}
\def \D {\mathbb{D}}
\def \M {\mathcal M}
\def\G {\mathbb G}
\def \CN {\mathbb N}
\def \CO {\mathcal O}
\def \CR {\mathcal R}
\def \CS {\mathcal S}
\def \CW {\mathcal W}
\def \rk {{\mbox rank}}
 \def \Vol {\mathrm{Vol}}
\def \tr {\mathrm{tr}}

\def\frametitle {}
\def\block {}

\section{Introduction}

For large classes of links $K$, but not for all links, the colored Jones polynomial, a sequence of (Laurent-) polynomial link invariants $J_{K,N}(q)$ indexed by a natural number $N$, develops a well-defined tail: Up to a common sign change the first $N$ coefficients of $J_{K,N}(q)$ agree with the first $N$ coefficients of $J_{K,N+1}(q)$ for all $N$. This gives rise to a power series with interesting properties. For example, for many knots with small crossing numbers the tail functions are given by products of one-variable specializations of the two-variable Ramanujan theta function \cite{ArmondDasbach:RogersRamanujan}.

The colored Jones polynomial $J_{K^*,N}(q)$ of the mirror image $K^*$ of $K$ satisfies $J_{K^*,N}(q)=J_{K,N}(1/q).$
If it exist the tail function of $J_{K^*,N}(q)$ is called the head of $J_{K,N}(q)$.
It was shown in \cite{DL:VolumeIsh, DasbachLin:HeadAndTail, DT:RefinedBound} that the head and tail functions for alternating links contain geometric information that can be used to give upper and lower bounds for the hyperbolic volume of a non-torus alternating link. In a series of papers and in a book Futer, Kalfagianni and Purcell extended those results to larger and larger classes of links (e.g. \cite{FKP:VolumeJones, FKP:Guts}). It is the goal of this paper to show that one can express the head or tail functions of an adequate link, a large class of links that contain alternating knots, as products of head or tail functions of alternating links that can be read-off an adequate diagram of the link. As a geometric application we strengthen a Theorem of Futer, Kalfagianni and Purcell related to the fiberedness of an adequate link.

\begin{example}
Let $K^*$ be the mirror image of the non-alternating knot $K=10_{154}$ as in Figure \ref{Example:10.154mirror}.  Its colored Jones polynomial $J_{K^*,N}(q)$ - up to multiplication with a suitable power $\pm q^{s_{N}}$ for some integers $s_{N}$ - is given by:

\begin{center}
\begin{tabular}{|l|l|}
\hline
$N=2$ & $1-2 q+2 q^2-3 q^3+2 q^4+\ldots $\\
$N=3$ & $1-2 q-q^2+5 q^3-3 q^4-4 q^5 +\ldots $\\
$N=4$ & $1-2 q-q^2+2 q^3+4 q^4-2 q^5-7 q^6 +\ldots$\\
$N=5$ & $1-2 q-q^2+2 q^3+q^4+5 q^5-6 q^6-5 q^7 +\ldots$\\
$N=6$ & $1-2 q-q^2+2 q^3+q^4+2 q^5+q^6-4 q^7-7 q^8 + \ldots$\\
$N=7$ & $1-2 q-q^2+2 q^3+q^4+2 q^5-2 q^6+3 q^7-6 q^8-7 q^9 + \ldots$\\
$N=8$ & $1-2 q-q^2+2 q^3+q^4+2 q^5-2 q^6+q^8-6 q^9-4 q^{10}+2 q^{11} + \ldots$\\
\hline
\end{tabular}  
\end{center} 
The data was obtained from Dror Bar-Natan's Mathematica package KnotTheory \cite{BarNatan:KnotTheory}.
Thus the tail of the colored Jones polynomial of the mirror image of the knot $10_{154}$ is given by the series: $$1-2 q - q^2+ 2q^3+q^4+ 2 q^5- 2 q^6+ 0 q^7 + \dots$$
\end{example}

We will study the tail series for adequate links, a class that generalizes alternating links. 
We will show that for every adequate link there is a prime alternating link with coinciding tail function.
Furthermore, we will strengthen a theorem of Futer, Kalfagianni and Purcell \cite{FKP:Guts, Futer:Fiberedness} and relate the complete tail function of the colored Jones polynomial of an adequate knot to the fiberedness of the link with fiber surface a certain spanning surface of that link.

\bigskip

\noindent {\bf Acknowledgment: } The authors thank Effie Kalfagianni for helpful suggestions and discussions during a visit to LSU.

\section{The Main Theorem}

\subsection{The all-$A$ state surface and the all-$A$ graph}

Let $K$ be a link with link diagram $D$ with $c$ crossings; to each crossing one can assign either of two Kauffman smoothings as in Figure \ref{Kauffman smoothings}.

\begin{figure}[ht] 
\begin{tikzpicture}[baseline=0cm, scale=0.8]
\draw[style=thick] (135:1)--(315:1);
\draw[color=white, line width=5pt] (45:1) -- (225:1);
\draw[style=thick] (45:1) -- (225:1);
\draw[style=thick, ->]  (1.4,0.2)  -- (2.4,0.6) node[above] {$A$} -- (3.4,1);
\draw[style=thick,->] (1.4,-0.2) -- (2.4,-0.6) node[below] {$B$} -- (3.4,-1);

\draw[xshift=4.8cm, yshift=-1cm, style=thick] (210:1) .. controls (240:0.1) and (300:0.1) .. (330:1);
\draw[xshift=4.8cm, yshift=-1cm, style=thick] (150:1) .. controls (120:0.1) and (60:0.1) .. (30:1);
\draw[xshift=4.8cm, yshift=1cm, style=thick] (300:1) .. controls (330:0.1) and (390:0.1) .. (420:1);
\draw[xshift=4.8cm, yshift=1cm, style=thick] (240:1) .. controls (210:0.1) and (150:0.1) .. (120:1);
\end{tikzpicture}
\caption{$A$ and $B$ smoothings for a link diagram  \label{Kauffman smoothings}}
\end{figure}
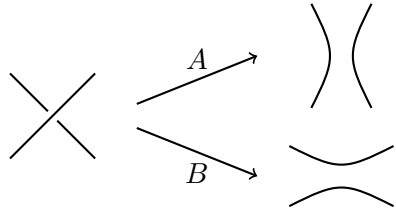

Thus there are $2^c$ ways, called states, to assign smoothings to the $c$ crossings of the diagram. Two of those states are important to us: The all-$A$ state and the all-$B$ state, where either only $A$-smoothings or only $B$-smoothings are assigned to the crossings. States are represented by smoothing diagrams in the plane.
Figure \ref{allB of 10_154} shows the all-$A$ smoothing diagram for a diagram of the mirror image of the knot $10_{154}$. Note, that a link can be recovered from its all-$A$ or all-$B$ smoothing diagrams.

\begin{figure}[ht] 
\includegraphics[width=5cm]{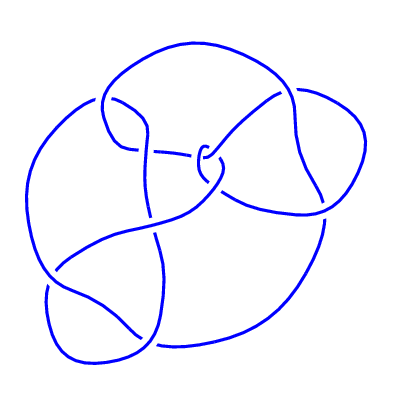}
\hspace {0.5cm}
\begin{tikzpicture}[baseline=-2.4cm, scale=0.8]
\draw[style=thick] (0,0)  circle (2cm);
\draw[style=thick] (0.8,0) circle (0.6cm);
\draw[style=thick] (-1,0) circle (0.4cm);
\draw[style=thick] (-3.5,1.3) circle (0.7);
\draw[style=thick] (-3.5,-1.3) circle (0.7);
\draw[style=dashed] (-0.9,1.8) -- (-3,1.8);
\draw[style=dashed] (-1.35, 1.5)-- (-2.8,1.5);
\draw[style=dashed] (-1.35, -1.5)-- (-2.8,-1.5);
\draw[style=dashed] (-3.5, -0.6) -- (-3.5, 0.6);
\draw[style=dashed] (0.2,0)--(-0.6,0);
\draw[style=dashed] (-1.4,0)--(-2,0);
\draw[style=dashed] (-3.35, 0.55)--(-1.4,-1.45);
\draw[style=dashed] (0.22,-0.1) -- (-1.64, -0.8);
\draw[style=dashed] (1.5,0.2)--(1.95,0.2);
\draw[style=dashed] (1.5,-0.2)--(1.95,-0.2);
\end{tikzpicture}
\caption{A diagram of the mirror image of the knot $10_{154}$ and its all-$A$ smoothing diagram. \label{Example:10.154mirror}}
\end{figure}
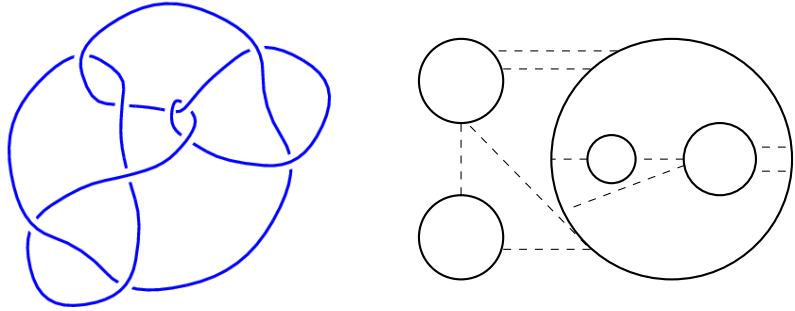

\begin{figure}[ht]
\begin{tikzpicture}[baseline=0cm, scale=0.8]
\draw[style=thick] (0,0)--(210:2)--(150:2)--(0,0)--(30:2)--(-30:2)--(0,0);
\draw[style=thick] (150:2) .. controls (110:1) and (-70:1) .. (-30:2);
\draw[style=thick] (150:2) .. controls (190:1) and (10:1) .. (-30:2);
\fill (0,0) circle(0.07);
\fill (210:2) circle(0.07);
\fill (150:2) circle (0.07);
\fill (30:2) circle (0.07);
\fill (-30:2) circle (0.07);
\end{tikzpicture}
\caption{The graph $\G_A$}
\end{figure}

The  all-$A$ smoothing and all-$B$-smoothing diagram  naturally lead to two graphs $\G_A$ and $\G_B$ where the vertices are the circles of the smoothing diagrams and the edges correspond to the 
smoothed crossing. A link diagram is called $A$-adequate (or $B$-adequate) if $\G_A$ (or $\G_B$) does not contain a loop, i.e. an edge that connects a vertex to itself. A link is adequate if it admits an $A$-adequate as well as a $B$-adequate diagram.

It was shown in \cite{Armond:HeadAndTailConjecture} that the tail of the colored Jones polynomial of an $A$-adequate link exists.
For alternating knots this was independently shown by Garoufalidis and Le \cite{GL:NahmSums}  and generalized to higher order tails.
An approach via link homologies was given by Rozansky \cite{Rozansky:HeadAndTail}.

\begin{maintheorem} \label{Main Theorem}
Suppose two $A$-adequate link diagrams only differ locally in their all-$A$ smoothing diagram as in Figure  \ref{EqualTails}

\begin{figure}[ht] 
\begin{tikzpicture}[scale=1]
\draw[style=thick] (0,0) -- (0,3);
\draw[style=dashed] (0,1)--(1,1);
\draw[style=thick] (-1,0)--(-1,3);
\draw[style=dashed] (-1,2)--(0,2);
\draw[style=thick] (1,0)--(1,3);
\draw[<->] (2,1.5)--(3,1.5);
\draw[style=thick] (5,0) -- (5,3);
\draw[style=dashed] (4,1)--(5,1);
\draw[style=thick] (4,0) -- (4,3);
\draw[style=dashed] (5,2)--(6,2);
\draw[style=thick] (6,0) -- (6,3); 
\end{tikzpicture}
\caption{Two links with coinciding tails of colored Jones polynomial \label{EqualTails}}
\end{figure}
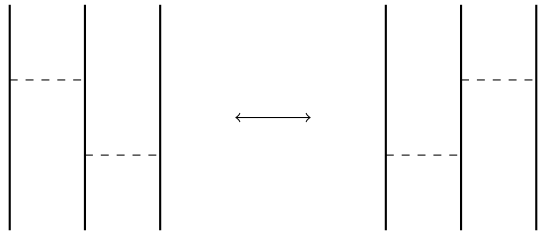

\noindent
where the three vertical lines represent arcs in three different smoothed circles. Then the tails of the colored Jones polynomials of the corresponding links coincide.
\end{maintheorem}

\begin{remark} The conditions on the $A$-adequate link diagram in the Main Theorem \ref{Main Theorem} imply that the link diagram is not alternating. For from the three circles in the all-$A$ smoothing depicted in Figure \ref{EqualTails} either the left or the right circle has to lie inside the middle circle and the other circle has to lie outside. In the all-$A$ smoothing of an alternating diagram for a given circle all other circles either lie outside or inside that circle.
\end{remark}

The proof of the Main Theorem is given in Section \ref{Sec:Main}.
First we will develop a few Corollaries of the Main Theorem.

The first corollary shows that the tail of the colored Jones polynomial of an $A$-adequate link is fully determined by the graph $\G_A$.
More specifically, a reduced graph $\G_A'$ is constructed from $\G_A$ by replacing all parallel edges, i.e. edges that connect the same two vertices, by a single edge. Figure \ref{reduced10154} gives an example. Then:

\begin{cor} \label{Dependence on abstract graph}
For an $A$-adequate link with $A$-adequate diagram $D$ and reduced all-$A$-graph $\G_A'$ the tail of the colored Jones polynomial only depends on $\G_A'$. 
Moreover, for each $A$-adequate link $L$ there is a prime alternating link $L'$ such that the tails of the colored Jones polynomials of the two links coincide.
\end{cor}
 
\begin{proof}  
First assume that the link is alternating. Then the circles in the all-$A$ smoothing diagram trace out faces in the diagram of the knot.
Thus $G_A$ is a plane graph. In \cite{ArmondDasbach:RogersRamanujan} it is shown that the reduction from $\G_A$ to $\G_A'$ does not change the tail of the colored Jones polynomial. Furthermore, by Whitney's classification theorem all embeddings of a planar graph into the plane are related by $2$-isomorphisms.
Those correspond to mutations in the link diagram. Since the colored Jones polynomial is invariant under mutations (e.g. \cite{StT:MutationCJP}) the claim holds for alternating links.

If the link diagram is adequate but not alternating then there is a circle in the all-$A$ smoothing diagram such that on either side of the circle are other circles.
Take an inner-most circle of that form and divide it into two arcs $\alpha_1$ and $\alpha_2$. By the Main Theorem  the link can be transformed into a link with equal tail by moving all edges coming from crossings inside the circle to $\alpha_1$ and all edges coming from crossings outside the circle to $\alpha_2$.
Figure \ref{Example moving edges} gives an example.
The resulting link is a connected sum of an alternating link (inside the circle) and an adequate link with fewer crossings (outside the circle).
Since the colored Jones polynomial is multiplicative under connected sum the first claim follows.
In particular a tail function of the colored Jones polynomial of an adequate link can be expressed as the product of tail functions of colored Jones polynomials of alternating links. As shown in \cite{ArmondDasbach:RogersRamanujan} the product of the tail functions of two prime alternating links is again the tail function of a prime alternating link and the second claim follows.
\end{proof} 

\begin{example}
Figure \ref{Example moving edges} illustrates the operation in the proof of Corollary \ref{Dependence on abstract graph} for the all-$A$ smoothing diagram of the mirror image of the knot $10_{154}$.

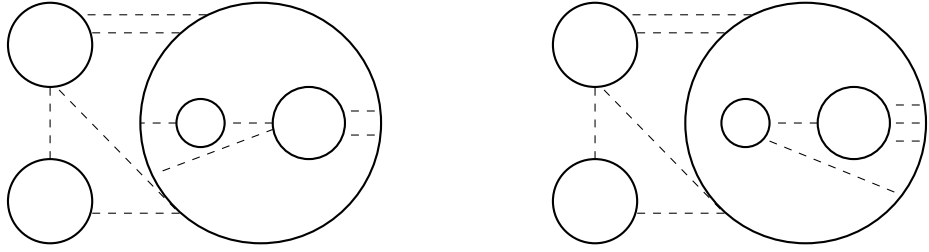
\begin{figure}[ht] 
\begin{tikzpicture}[baseline=-2.4cm, scale=0.8]
\draw[style=thick] (0,0)  circle (2cm);
\draw[style=thick] (0.8,0) circle (0.6cm);
\draw[style=thick] (-1,0) circle (0.4cm);
\draw[style=thick] (-3.5,1.3) circle (0.7);
\draw[style=thick] (-3.5,-1.3) circle (0.7);
\draw[style=dashed] (-0.9,1.8) -- (-3,1.8);
\draw[style=dashed] (-1.35, 1.5)-- (-2.8,1.5);
\draw[style=dashed] (-1.35, -1.5)-- (-2.8,-1.5);
\draw[style=dashed] (-3.5, -0.6) -- (-3.5, 0.6);
\draw[style=dashed] (0.2,0)--(-0.6,0);
\draw[style=dashed] (-1.4,0)--(-2,0);
\draw[style=dashed] (-3.35, 0.55)--(-1.4,-1.45);
\draw[style=dashed] (0.22,-0.1) -- (-1.64, -0.8);
\draw[style=dashed] (1.5,0.2)--(1.95,0.2);
\draw[style=dashed] (1.5,-0.2)--(1.95,-0.2);
\end{tikzpicture}
\hspace*{2cm}
\begin{tikzpicture}[baseline=-2.4cm, scale=0.8]
\draw[style=thick] (0,0)  circle (2cm);
\draw[style=thick] (0.8,0) circle (0.6cm);
\draw[style=thick] (-1,0) circle (0.4cm);
\draw[style=thick] (-3.5,1.3) circle (0.7);
\draw[style=thick] (-3.5,-1.3) circle (0.7);
\draw[style=dashed] (-0.9,1.8) -- (-3,1.8);
\draw[style=dashed] (-1.35, 1.5)-- (-2.8,1.5);
\draw[style=dashed] (-1.35, -1.5)-- (-2.8,-1.5);
\draw[style=dashed] (-3.5, -0.6) -- (-3.5, 0.6);
\draw[style=dashed] (0.2,0)--(-0.6,0);
\draw[style=dashed] (-0.6, -0.3) -- (1.6, -1.2);
\draw[style=dashed] (-3.35, 0.55)--(-1.4,-1.45);
\draw[style=dashed] (1.5,0.3)--(1.95,0.3);
\draw[style=dashed]  (1.5, 0)--(1.95, 0);
\draw[style=dashed] (1.5,-0.3)--(1.95,-0.3);
\end{tikzpicture}
\caption{The all-$A$ smoothing diagram. \label{allB of 10_154} of the mirror image of the knot $10_{154}$ and the all-$A$ smoothing diagram of a link with identical tail function \label{Example moving edges}}
\end{figure}

\begin{figure}[ht]
\subfloat[$\G_A$]{
\begin{tikzpicture}[baseline=0cm, scale=0.8]
\draw[style=thick] (0,0)--(210:2)--(150:2)--(0,0)--(30:2)--(-30:2)--(0,0);
\draw[style=thick] (150:2) .. controls (110:1) and (-70:1) .. (-30:2);
\draw[style=thick] (150:2) .. controls (190:1) and (10:1) .. (-30:2);
\fill (0,0) circle(0.07);
\fill (210:2) circle(0.07);
\fill (150:2) circle (0.07);
\fill (30:2) circle (0.07);
\fill (-30:2) circle (0.07);
\end{tikzpicture}
}
\hspace*{1cm}
\subfloat[$\G_A'$]
{
\begin{tikzpicture}[baseline=0cm, scale=0.8]
\draw[style=thick] (0,0)--(210:2)--(150:2)--(0,0)--(30:2)--(-30:2)--(0,0);
\fill (0,0) circle(0.07);
\fill (210:2) circle(0.07);
\fill (150:2) circle (0.07);
\fill (30:2) circle (0.07);
\fill (-30:2) circle (0.07);
\end{tikzpicture}
}
\caption{The graph $\G_A$ and the reduced graph $\G_A'$ \label{reduced10154}}
\end{figure}
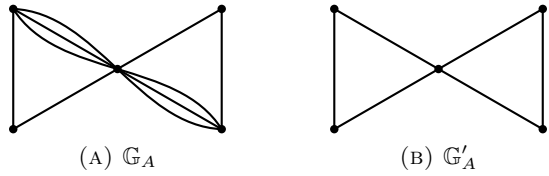

In particular, Figure \ref{reduced10154} shows that the tail function of the knot $10_{154}$ is equivalent to the tail function of the connected sum of two (negative) trefoils.
It was shown in \cite{ArmondDasbach:RogersRamanujan} that those tails are express as an evaluation of the two variable Ramanujan theta function:
$$f(a,b)=\sum_{n=-\infty}^{\infty} a^{n(n+1)/2} b^{n (n-1)/2}.$$
More specifically: The tail function of the (negative) trefoil is
$$f(-q^2,-q)= \sum_{n=-\infty}^{\infty} (-q)^{3/2 n^2+n/2}.$$
Thus the tail function of the knot $10_{154}$ is $f(-q^2,-q)^2$.

\end{example}

The main theorem immediately implies the following extension of a theorem of Purcell, Kalfagianni and Futer \cite{FKP:Guts, Futer:Fiberedness}. For every state one can construct a spanning surface for the link similar to Seiferts construction of Seifert surfaces (e.g. \cite{Futer:Fiberedness}). Let $S_A$ be the state surface for the all-$A$ state and $\G_A'$ as constructed above.

\begin{theorem}
For a link $L$ with diagram $D_L$ and state surface $S_A$ determined by the all-$A$ state the following statements are equivalent: 

\begin{enumerate}
\item $S^3-L$ fibers over $S^1$, with fiber $S_A$. \label{Equiv1}
\item The reduced all-$A$ graph $\G_A'$ is a tree. \label{Equiv2}
\item The diagram $D_L$ is $A$-adequate and $\beta_A=0$.  \label{Equiv3}
\item The diagram $D_L$ is $A$-adequate and the tail of the colored Jones polynomial is $1$.    \label{Equiv4}
\end{enumerate}
\end{theorem}

\begin{proof}
The equivalence of  (\ref{Equiv1}) and (\ref{Equiv2}) is shown in \cite{FKP:Guts}. By \cite{DasbachLin:HeadAndTail} and \cite{DL:VolumeIsh} it follows that (\ref{Equiv2}) and (\ref{Equiv3}) are equivalent. Corollary \ref{Dependence on abstract graph} establishes the equivalence of (\ref{Equiv4}) and (\ref{Equiv2}).
\end{proof}

Moreover, it directly follows from Corollary \ref{Dependence on abstract graph}

\begin{cor}[\cite{Armond:Walks}]
For a closed positive braid the tail of the colored Jones polynomial is identically $1$.
\end{cor}

\section{Proof of Main Theorem} \label{Sec:Main}

\subsection{Skein Theory} For a more detailed explanation of skein theory, see  \cite{Lickorish:KnotTheoryBook, MasbaumVogel:3valentGraphs}.

The Kauffman bracket skein module, $S(M;R,A)$, of a $3$-manifold $M$ and ring $R$ with invertible element $A$, is the free $R$-module generated by isotopy classes of framed links in $M$, modulo the submodule generated by the Kauffman relations:

\begin{center}
\begin{tabular}{c c}
$\includegraphics[width=.45in]{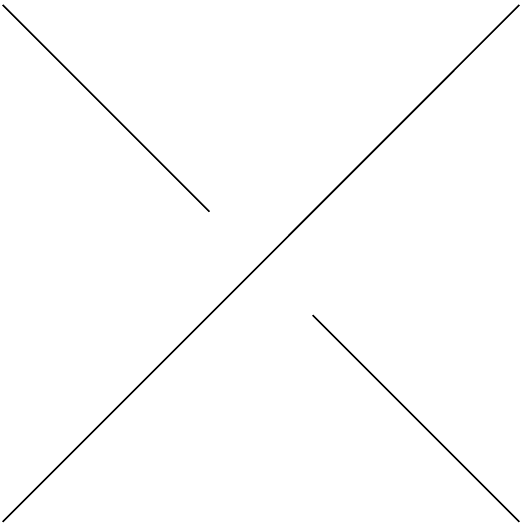} \raisebox{13pt}{$\;= A$} \includegraphics[width=.45in]{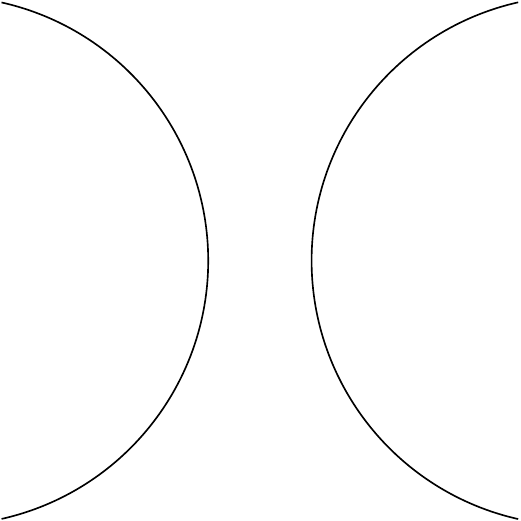} \raisebox{13pt}{$+ A^{-1}$} \includegraphics[width=.45in]{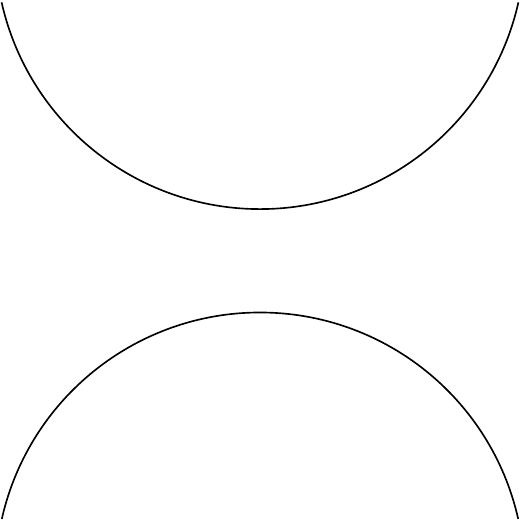}\;,\qquad\qquad$
&
$\includegraphics[width=.45in]{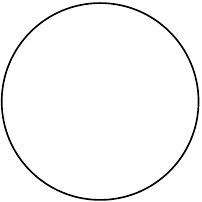} \raisebox{13pt}{$\;= -A^2 -A^{-2}$}$
\end{tabular}
\end{center}

If $M$ has designated points on the boundary, then the framed links must include arcs which meet all of the designated points.

In this paper we will take $R = \mathbb{Q}(A)$, the field of rational functions in variable $A$ with coefficients in $\mathbb{Q}$. As we are concerned with the lowest terms of a polynomial, we will need to express rational functions as Laurent series and define:

\begin{df}
Let $f \in \Q (A)$, define $d(f)$ to be the minimum degree of $f$ expressed as a Laurent series in $A$.
\end{df}

Note that $d(f)$ can be calculated without referring to the Laurent series. Any rational function $f$ expressed as $\frac{P}{Q}$ where $P$ and $Q$ are both polynomials. Then $d(f) = d(P) - d(Q)$.

\begin{df}
For two Laurent series $P_1(A)$ and $P_2(A)$ we define
$$P_1(A) \;\dot{=}_n\; P_2(A)$$
if after multiplying $P_1(A)$ by $\pm A^{s_1}$ and $P_2(A)$ by $\pm A^{s_2}$, $s_1$ and $s_2$ some powers, to get power series $P'_1(A)$ and $P'_2(A)$ each with positive constant term, $P'_1(A)$ and $P'_2(A)$ agree $\mod A^n$.
\end{df}

 We will be concerned with two particular skein modules: $S(S^3;R,A)$, which is isomorphic to $R$ under the isomorphism sending the empty link to $1$, and $S(D^3;R,A)$, where $D^3$ has $2n$ designated points on the boundary. With these designated points, $S(D^3;R,A)$ is also called the Temperley-Lieb algebra $TL_n$.

We will give an alternate explanation for the Temperley-Lieb algebra. First, consider the disk $D^2$ as a rectangle with $n$ designated points on the top and $n$ designated points on the bottom. Let $TLM_n$ be the set of all crossing-less matchings on these points, and define the product of two crossing-less matchings by placing one rectangle on top of the other and deleting any components which do not meet the boundary of the disk. With this product, $TLM_n$ is a monoid, which we shall call the Temperley-Lieb monoid.

Any element in $TL_n$ has the form $\sum_{M \in TLM_n} c_M M$, where $c_M \in \mathbb{Q}(A)$. Multiplication in $TL_n$ is slightly different from multiplication in $TLM_n$, because in $TLM_n$ complete circles are removed, but in $TL_n$ when a circle is removed, it is replaced with $(-A^2-A^{-2})$.

There is a special element in $TL_n$ of fundamental importance to the colored Jones polynomial, called the Jones-Wentzl idempotent, denoted $f^{(n)}$. Diagramatically this element is represented by an empty box with $n$ strands coming out of it on two opposite sides. By convention an $n$ next to a strand in a diagram indicates that the strand is replaced by $n$ parallel ones.

With $$\Delta_{n}:=(-1)^{n} \frac {A^{2(n+1)}-A^{-2 (n+1)}}{A^{2}-A^{-2}}$$ and
$\Delta_{n}!:= \Delta_{n} \Delta_{n-1} \dots \Delta_{1}$ the Jones-Wenzl idempotent satisfies

\begin{center}
\begin{tabular}{c c}
$\begin{tikzpicture}[baseline=0.8cm]
\draw(0,0)--(0,1) 
node[rectangle, draw, ultra thick, fill=white]{ } --(0,2) node[right]{\scriptsize{n+1}};
\end{tikzpicture} =
\begin{tikzpicture}[baseline=0.8cm]
\draw(0,0)--(0,1) 
node[rectangle, draw, ultra thick, fill=white]{ } --(0,2) node[right]{\scriptsize{n}};
\draw(0.7,0)--(0.7,2) node[right]{\scriptsize{1}};
\end{tikzpicture} -
\left( \frac {\Delta_{n-1}} {\Delta_{n}} \right ) \, \, 
\begin{tikzpicture}[baseline=0.8cm]
\draw (0,0) node[right]{\scriptsize{n}}--(0,1) node[left]{\scriptsize{n-1}}
--(0,2) node[right]{\scriptsize{n}};
\draw[ultra thick, fill=white] (-0.2 ,0.5) rectangle (0.4, 0.65); 
\draw (0.2, 0.65) arc (180:0:0.25);
\draw (0.7,0) node[right]{1}--(0.7, 0.65);
\draw[ultra thick, fill=white] (-0.2 ,1.5) rectangle (0.4, 1.35); 
\draw (0.2, 1.35) arc (-180:0:0.25);
\draw (0.7,2) node [right]{\scriptsize{1}} --(0.7, 1.35);
\end{tikzpicture},
\, \qquad \qquad \qquad
$
&
$\begin{tikzpicture}[baseline=0.8cm]
\draw(0,0)--(0,1) 
node[rectangle, draw, ultra thick, fill=white]{ } --(0,2) node[right]{\scriptsize{1}};
\end{tikzpicture} = \, \, 
\begin{tikzpicture}[baseline=0.8cm]
\draw(0,0)--(0,2) 
node[right]{\scriptsize{1}};
\end{tikzpicture}$
\end{tabular}
\end{center}

with the properties

$$\includegraphics[height=1in]{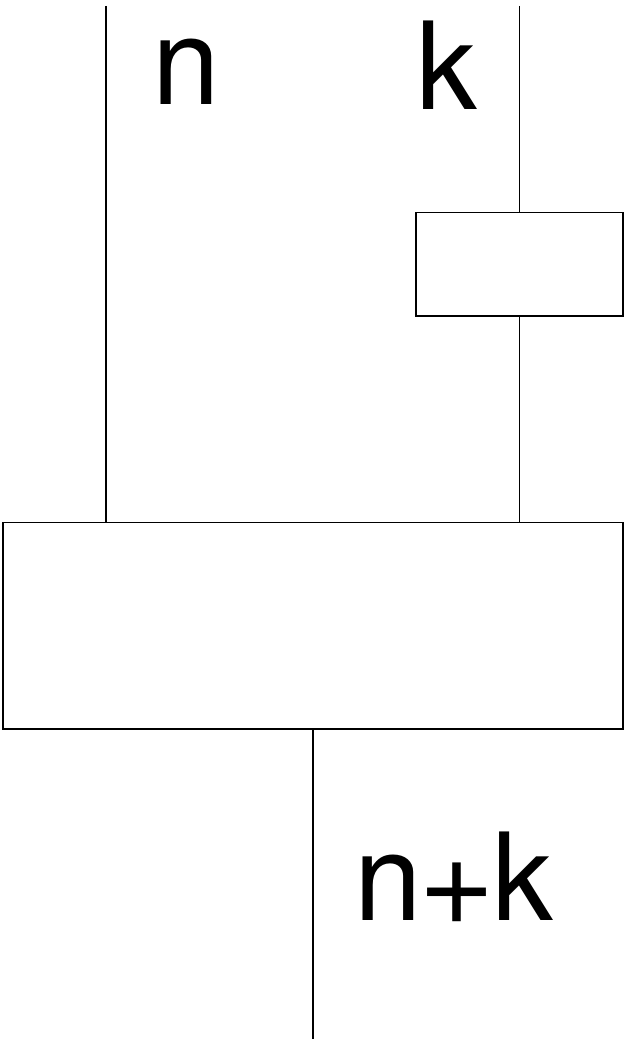} \raisebox{.5in}{$\;=\;$} \includegraphics[height=1in]{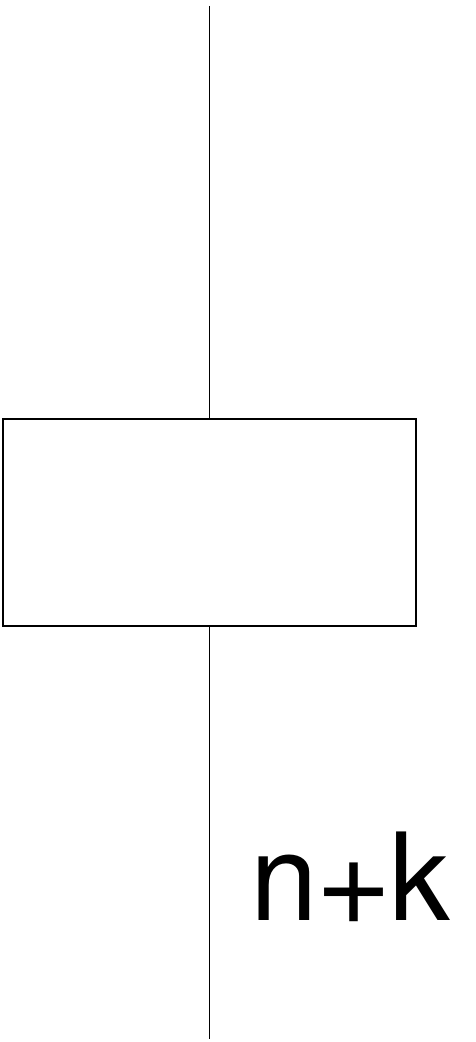}\raisebox{.5in}{,}
\, \qquad \qquad 
\includegraphics[height=1in]{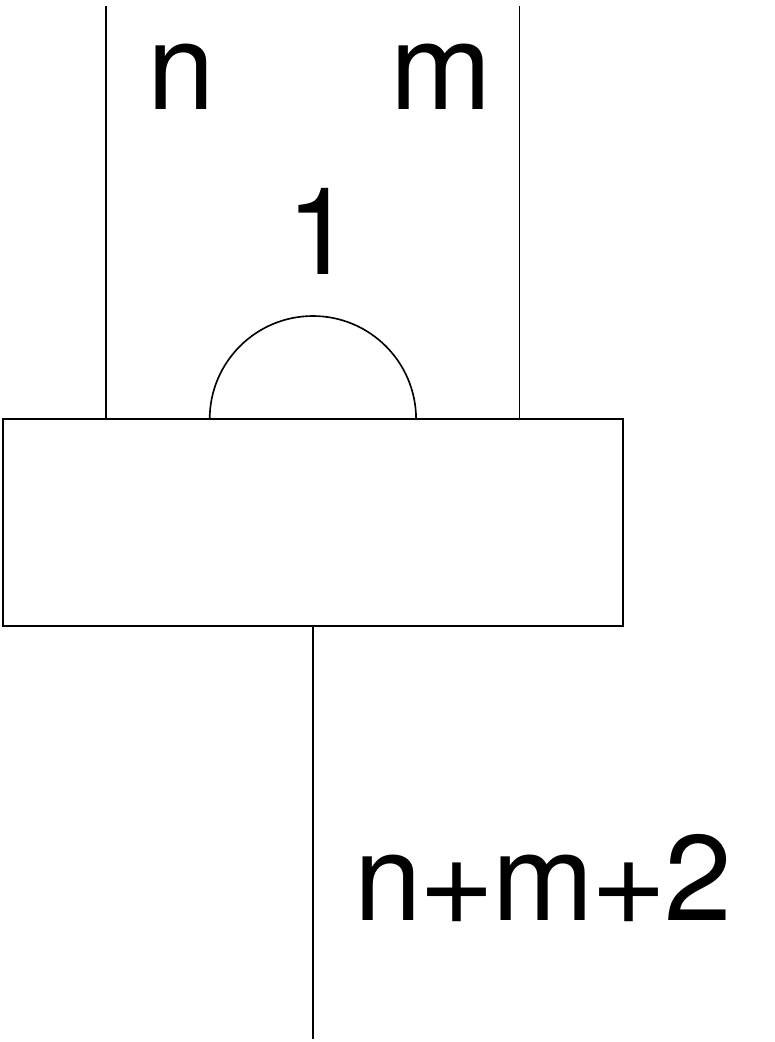} \raisebox{.5in}{$\;= 0$} 
\raisebox{.5in}{,}\, \qquad \qquad 
\includegraphics[height=1in]{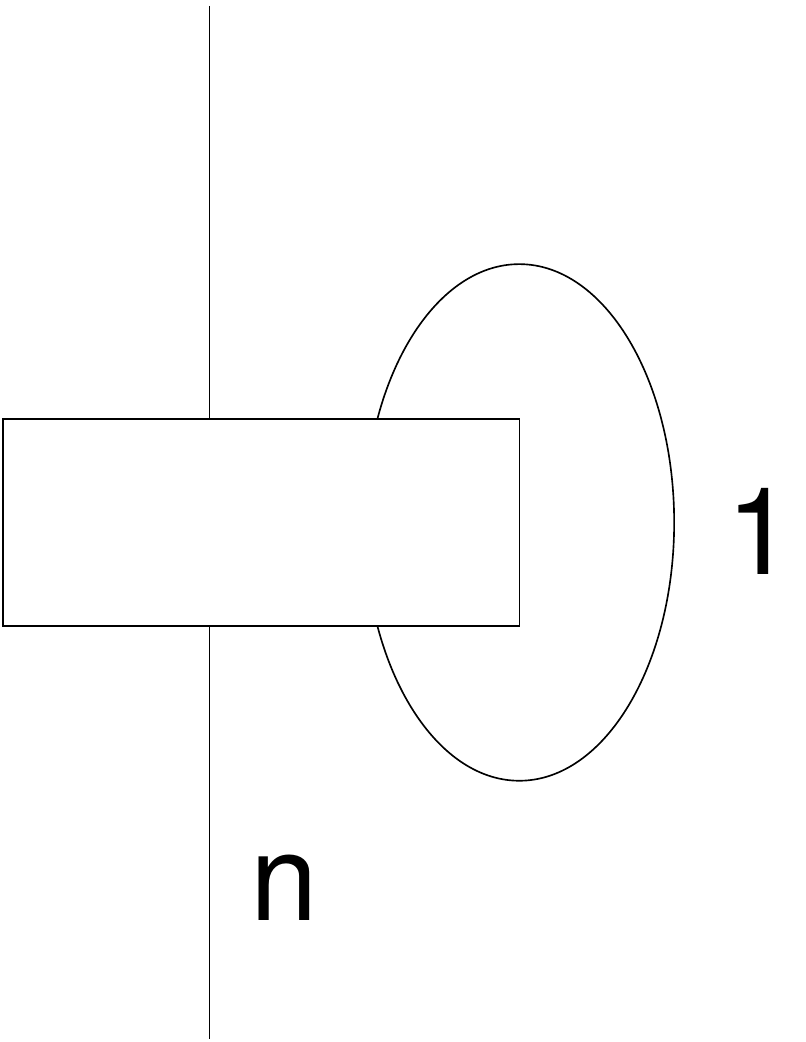} \raisebox{.45in}{$\;= \left( \frac {\Delta_{n+1}} {\Delta_{n}} \right )\;$} \includegraphics[height=1in]{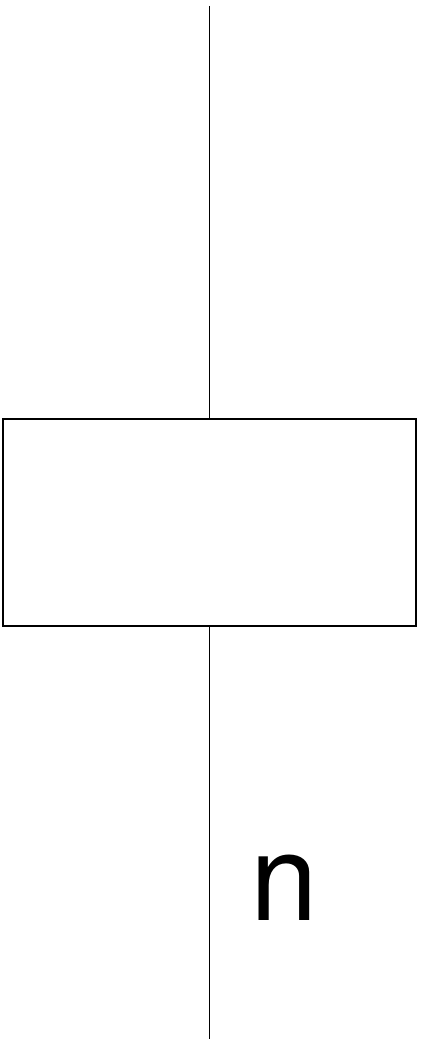}$$

The (unreduced) colored Jones polynomial $\tilde J_{n,D}(A)$ of a link diagram $D$ can be defined as the value of the skein relation applied to a diagram $D$, where every component is decorated by an $n$ together with the Jones-Wenzl idempotent.  Recall that $A^{-4}=q$. To obtain the reduced colored Jones polynomial of a link $L$ with diagram $D$, we must compensate for the writhe $w(D)$ of $D$ and divide by the value on the unknot. That is
$$J_{L,n+1}(q):=\left. (-A)^{(-n^2-2n)w(D)} \frac{\tilde J_{D,n}(A)} {\Delta_{n}}\right|_{A=q^{-1/4}}.$$

We can now properly define the tail of the colored Jones polynomial:

\begin{df}
The tail of the reduced colored Jones polynomial of a link $L$ -- if it exists -- is a series $T_L(q)$, with
$$T_L(q) \;\dot{=}_N\; J_{L,N}(q), \text{ for all } N$$

The tail of the unreduced colored Jones polynomial of a link $L$ -- if it exists -- is a series $\tilde T_L(A)$, with
$$\tilde T_L(A) \;\dot{=}_{4(n+1)}\; \tilde J_{D,n}(A), \text{ for all } n$$
\end{df}

In \cite{ArmondDasbach:RogersRamanujan} and \cite{Armond:HeadAndTailConjecture} we showed that the (unreduced) tail of $B$-adequate links are determined by a collection of crossingless skein diagrams coming from the B-state graph. Given a link diagram $D$, construct a skein diagram $S^{(n)}_B$ by replacing the former crossings in the all $B$-smoothing of $D$ by an idempotent colored $2n$ connecting the two circles.

\begin{figure}[htbp] %
   \centering
   \includegraphics[height=1.5in]{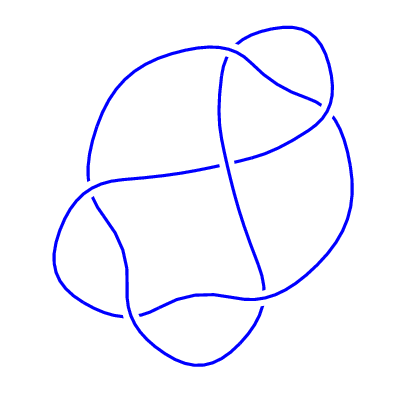}
   \quad
   \raisebox{32pt}{$\longrightarrow$}
   \quad
   \includegraphics[height=1.5in]{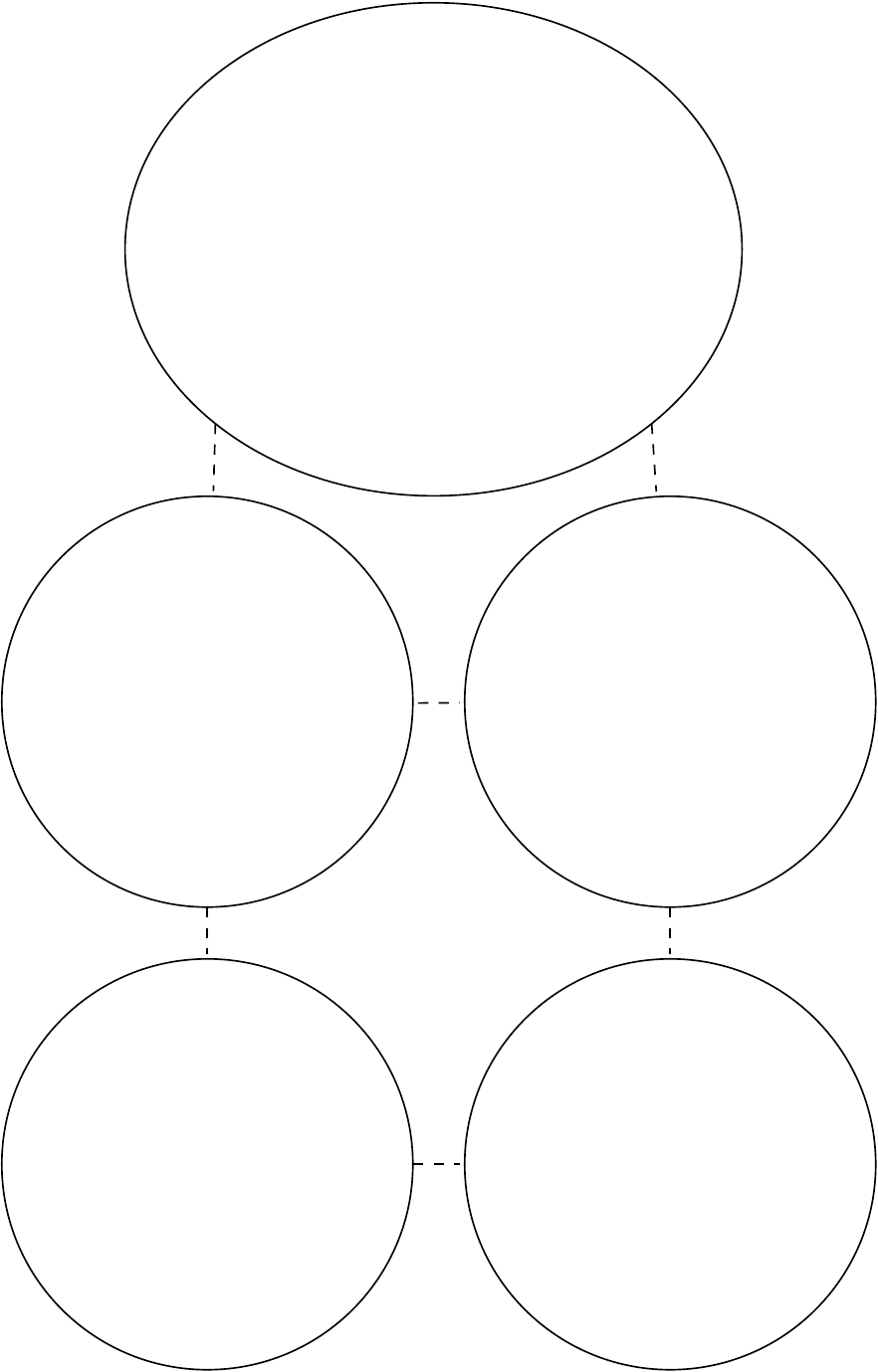}
   \quad
   \raisebox{32pt}{$\longrightarrow$}
   \quad
   \includegraphics[height=1.5in]{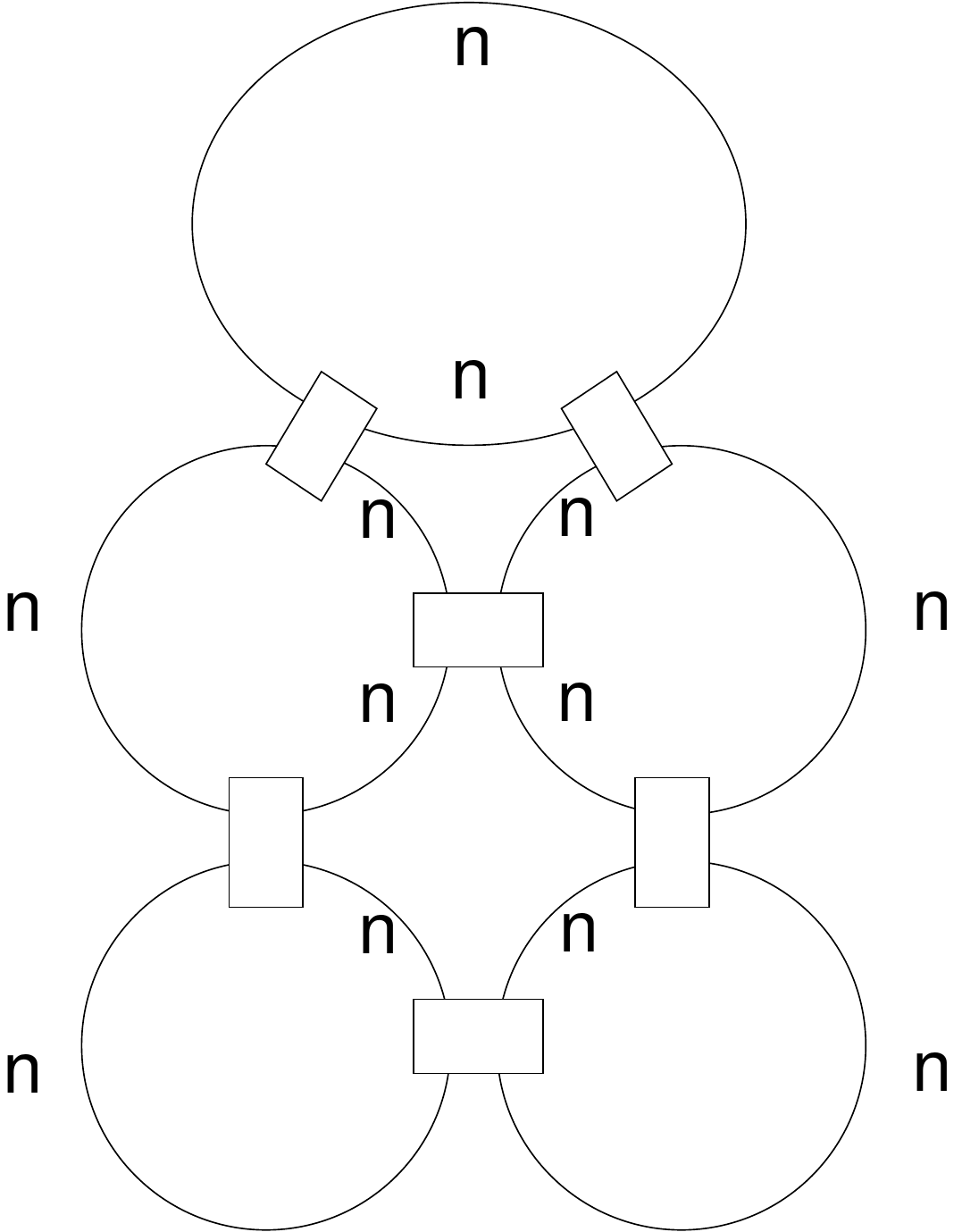}
  
   \caption{The diagram $S^{(n)}_B$ for $6_2$}
\end{figure}

\begin{lem}[\cite{Armond:HeadAndTailConjecture, ArmondDasbach:RogersRamanujan}]
\label{mainlemma}
If $D$ is a $B$-adequate link diagram, then 
$$\tilde J_{D,n}(A) \; \dot{=}_{4(n+1)} \; S^{(n)}_B.$$
\end{lem}

In \cite{Armond:HeadAndTailConjecture}, the first author found a lower bound for the minimum degree of any element of $S(S^3;R,A)$ which contains the Jones-Wenzl idempotent. Before we explain this bound, consider a crossing-less diagram $S$ in the plane consisting of arcs connecting Jones-Wenzl idempotents. We will define what it means for such a diagram to be adequate in much the same way that a knot diagram can be $A$- or $B$-adequate.

Construct a crossing-less diagram $\bar{S}$ from $S$ by replacing each of the Jones-Wenzl idempotents in $S$ by the identity of $TL_n$. Thus $\bar{S}$ is a collection of circles with no crossings. Consider the regions in $\bar{S}$ where the idempotents had previously been. $S$ is adequate if no circle in $\bar{S}$ passes through any one of these regions more than once. Figure \ref{fig:adequate} shows an example of a diagram that is adequate and Figure \ref{fig:inadequate} shows an example of a diagram that is not adequate. In both figures every arc is labelled $1$.

\begin{figure}[htbp]
	\centering
	\subfloat[An adequate diagram]{
	\includegraphics[width=2in]{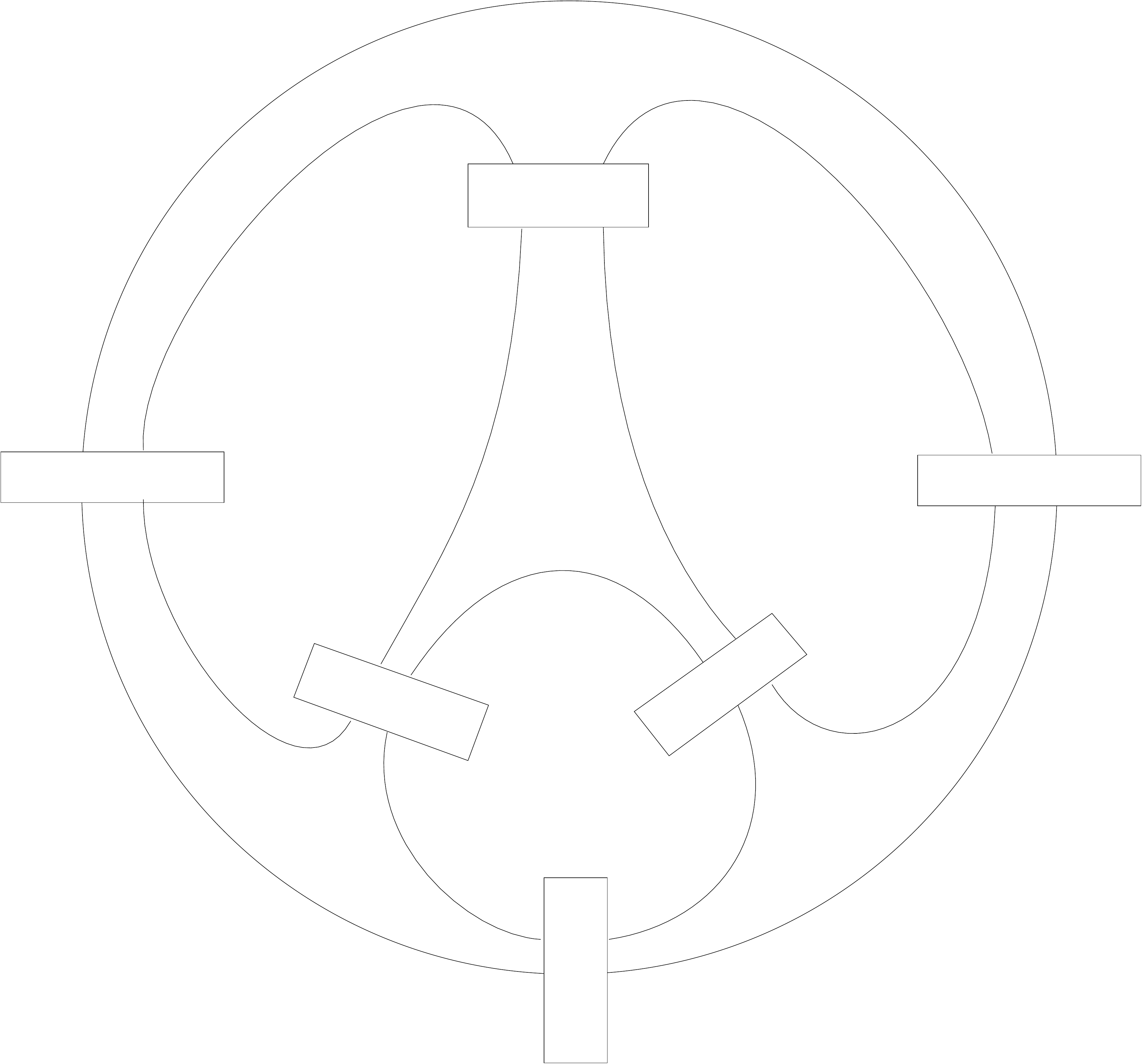}
	\label{fig:adequate}}
	\qquad
	\subfloat[An inadequate diagram]{
	 \includegraphics[width=2in]{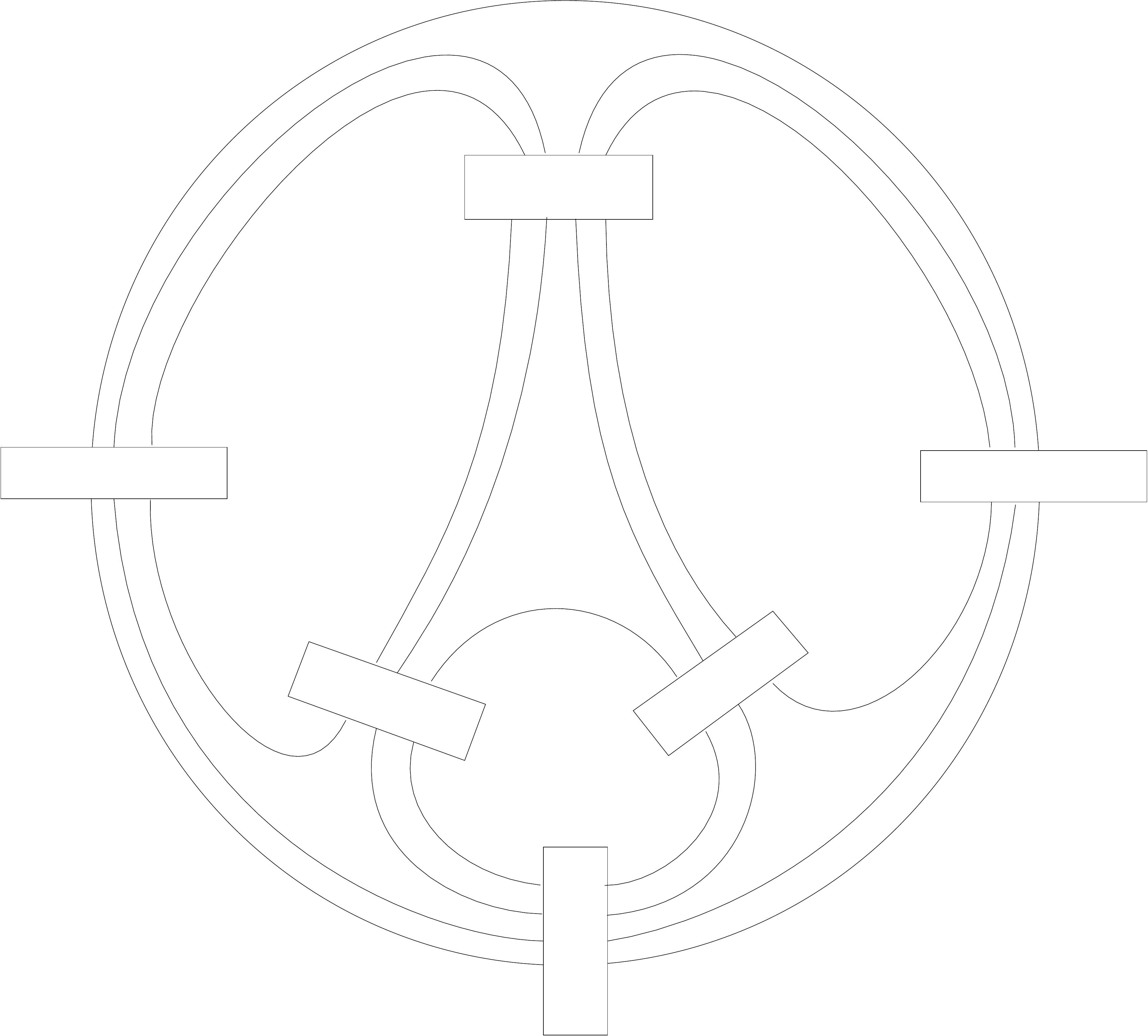}
	\label{fig:inadequate}}
	\caption{Example of adequate and inadequate diagrams}
	\label{fig:ExamplePiplup}
\end{figure}

If $S$ is adequate, then the number of circles in $\bar{S}$ is a local maximum, in the sense that if the idempotents in $S$ are replaced by other elements of $TLM_n$ such that there is exactly one hook total in all of the replacements, then the number of circles in this diagram is one less than the number of circles in $\bar{S}$. This is the key fact in the proof of the following lemma.

\begin{lem}[\cite{Armond:HeadAndTailConjecture}] 
\label{replace}
If $S \in S(S^3;R,A)$ is expressed as a single crossingless diagram containing Jones-Wenzl idempotents, then $d(S) \geq d(\bar{S})$.

If the diagram for $S$ is also an adequate diagram, then $d(S) = d(\bar{S})$.
\end{lem}

\subsection{Proof}
We now prove the Main Theorem by using Lemma \ref{mainlemma} and showing that if two $A$-adequate link diagrams $D_1$ and $D_2$ have all-$A$ smoothings that differ as in Main Theorem~\ref{Main Theorem}, then $S^{(n)}_{B_1} \;\dot{=}_{4(n+1)}\; S^{(n)}_{B_2}$. First, an important lemma:

\begin{lem}
\label{junkterms}
$$\includegraphics[height=.5in]{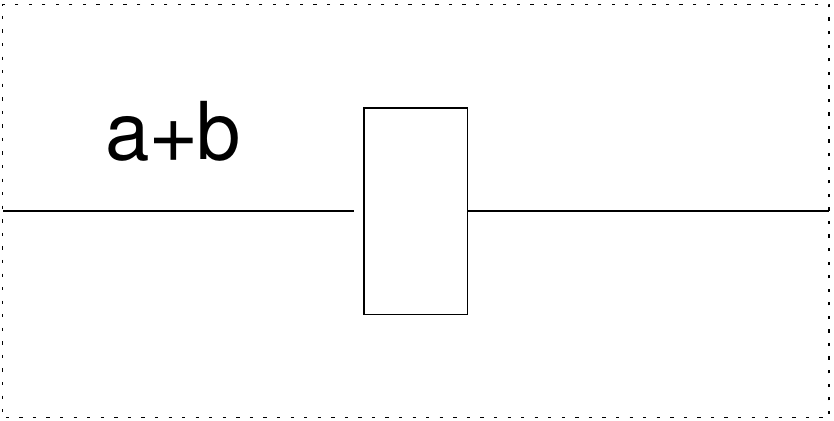} \raisebox{15pt}{$\; = \;$} \includegraphics[height=.5in]{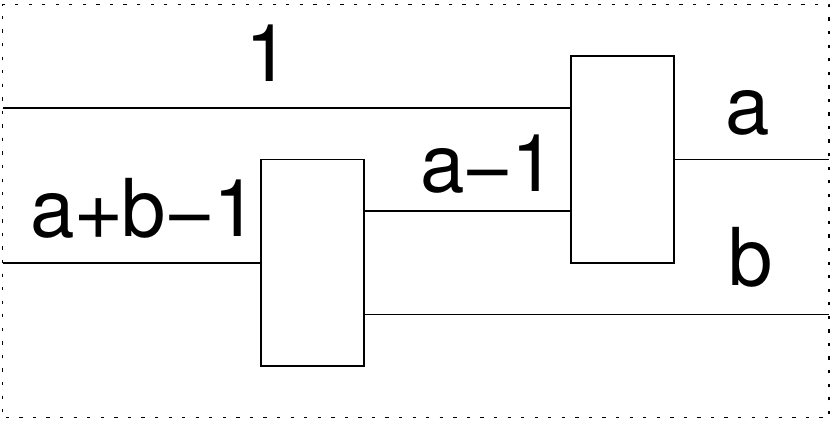} \raisebox{15pt}{$\; + \; (-1)^{a} \left(\frac{\Delta_{b-1}}{\Delta_{a+b-1}}\right)\;$} \includegraphics[height=.5in]{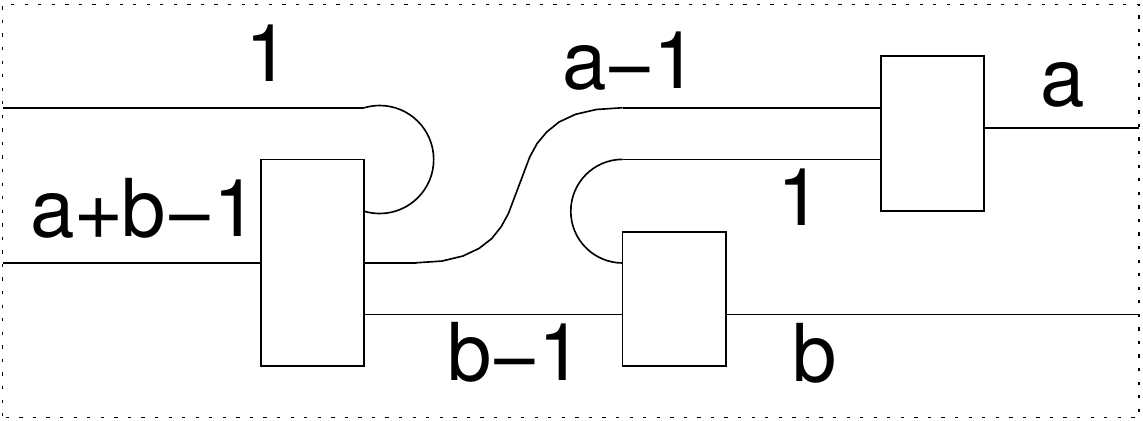}$$
\end{lem}

\begin{proof}
First we use the idempotent property to create a smaller idempotent on $a$ strands.
$$\includegraphics[height=.5in]{Skeinlem1.pdf} \raisebox{15pt}{$\; = \;$} \includegraphics[height=.5in]{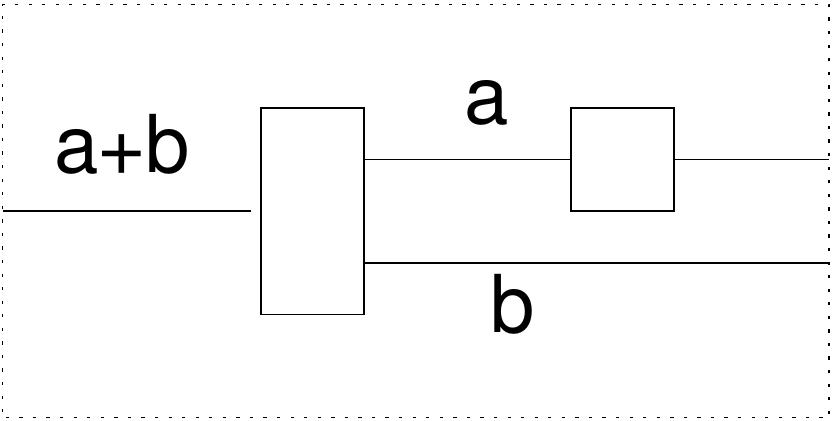}$$
Now we use the recursive relation on the larger idempotent.
$$\includegraphics[height=.5in]{SkeinlemP1.pdf} \raisebox{15pt}{$\; = \;$} \includegraphics[height=.5in]{Skeinlem2.pdf} \raisebox{15pt}{$\; - \left(\frac{\Delta_{a+b-2}}{\Delta_{a+b-1}}\right)\;$} \includegraphics[height=.5in]{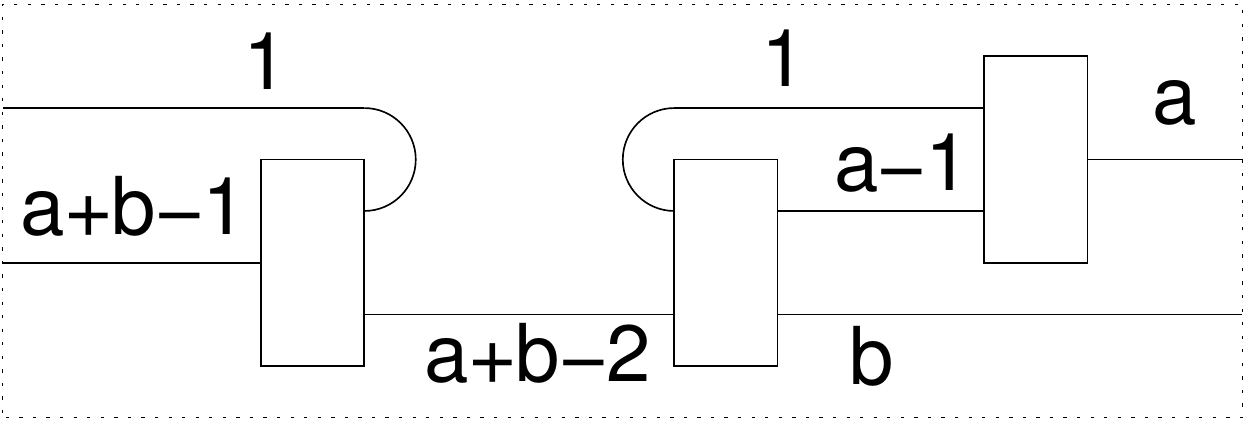}$$
If $a-1>0$, notice that when applying the recursive relation again on the right-most term, one of the two terms in the relation will be zero.
$$\includegraphics[height=.5in]{SkeinlemP2.pdf} \raisebox{15pt}{$\; = \;$} \includegraphics[height=.5in]{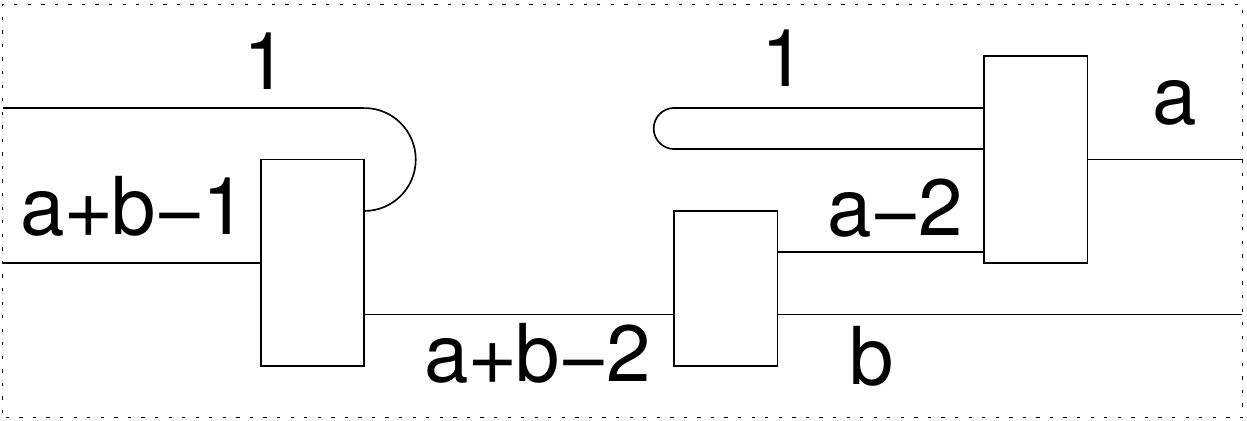} \raisebox{15pt}{$\; - \left(\frac{\Delta_{a+b-3}}{\Delta_{a+b-2}}\right)\;$} \includegraphics[height=.5in]{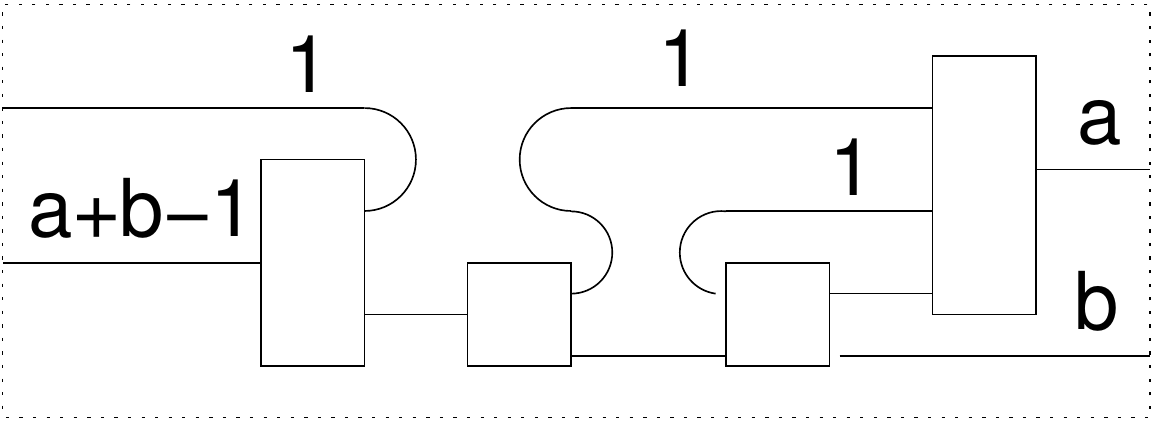}$$
Thus we get a simplification:
$$\includegraphics[height=.5in]{SkeinlemP2.pdf} \raisebox{15pt}{$\; = \;$}  \raisebox{15pt}{$\; - \left(\frac{\Delta_{a+b-3}}{\Delta_{a+b-2}}\right)\;$} \includegraphics[height=.5in]{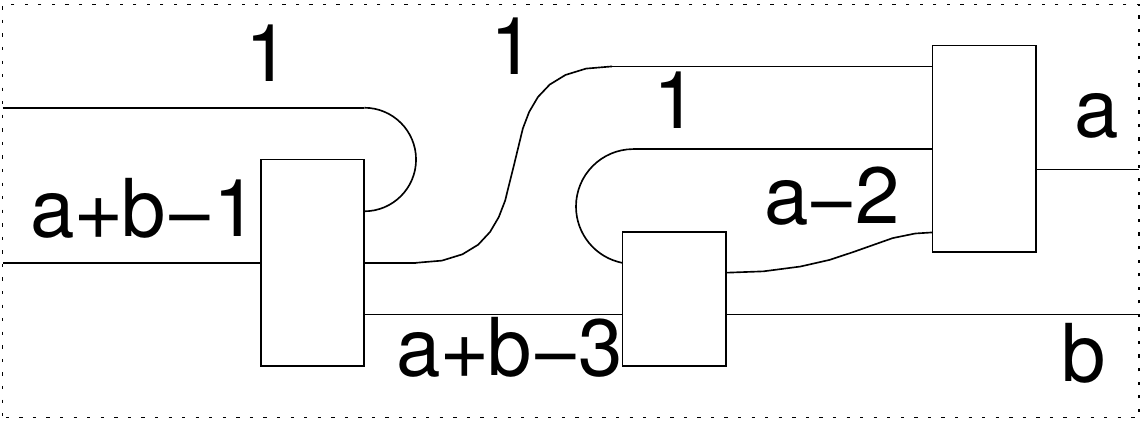}$$
By performing the recursion a total of $k$ times where $k \leq a-1$, we get the following equation:
$$\includegraphics[height=.5in]{SkeinlemP2.pdf} \raisebox{15pt}{$\; = \;$}  \raisebox{15pt}{$\; (-1)^k \left(\frac{\Delta_{a+b-2-k}}{\Delta_{a+b-2}}\right)\;$} \includegraphics[height=.5in]{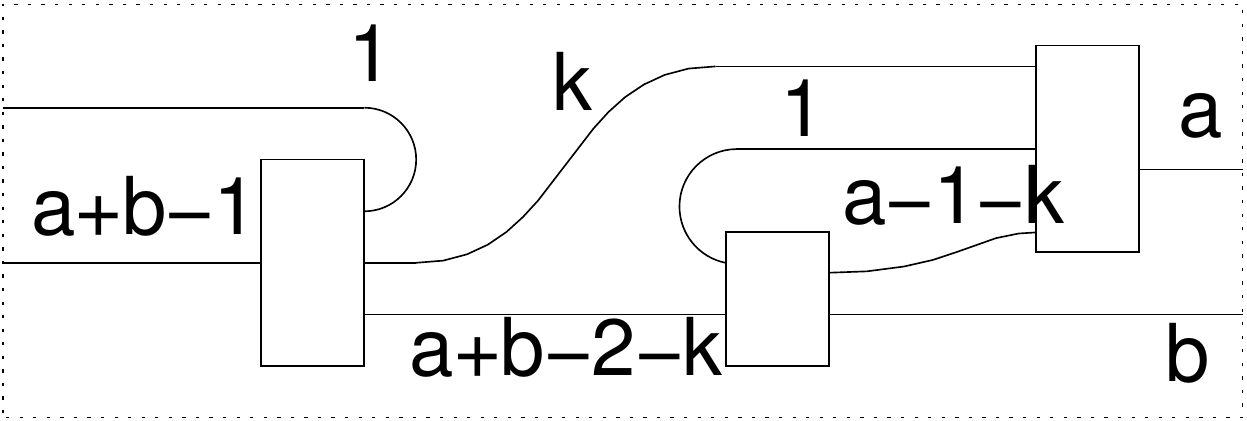}$$
Letting $k = a-1$, the lemma follows.\end{proof}

\begin{lem}
\label{big}

Given an element $C$ of $S(D^3;R,A)$ with $6$ points on the boundary colored $n$ such that the pairings of the $C$ with the left and right hand sides of the following equation are adequate diagrams in the plane, then we get the equation
$$\includegraphics[height=.5in]{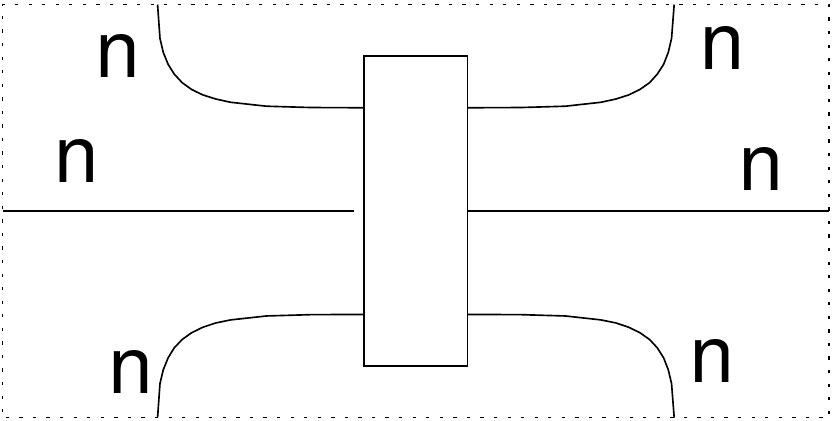} \raisebox{15pt}{$\; \dot{=}_{4(n+1)} \;$} \includegraphics[height=.5in]{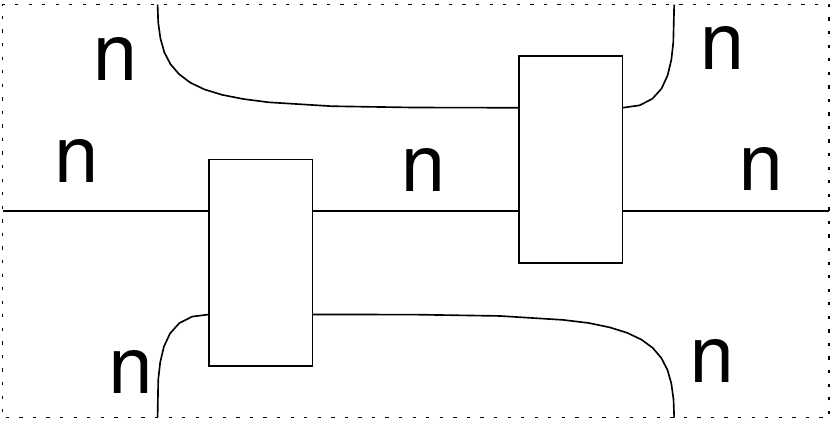} $$
where this equation and the equations appearing in the following proof are actually equations between the pairing of each term with the element $C$. That is the diagrams in this lemma are local pictures of diagrams in $S(S^3;R,A)$.
\end{lem}

\begin{remark}
The abuse of notation in this lemma is unambiguous because if each term is viewed as being an element of $S(D^3;R,A)$, then the equations do not make sense as the equivalence $\dot{=}_{4(n+1)}$ only applies to Laurent polynomials and Laurent series which elements of $S(D^3;R,A)$ with colored points on the boundary are not. Thus the equations only make sense if the terms actually represent elements of $S(S^3;R,A)$.
\end{remark}

\begin{proof} 
For any $k$ with $0 \leq k < n$, cosider the relation coming from Lemma~\ref{junkterms} with $a=2n-k$ and $b=n$:
$$\includegraphics[height=.5in]{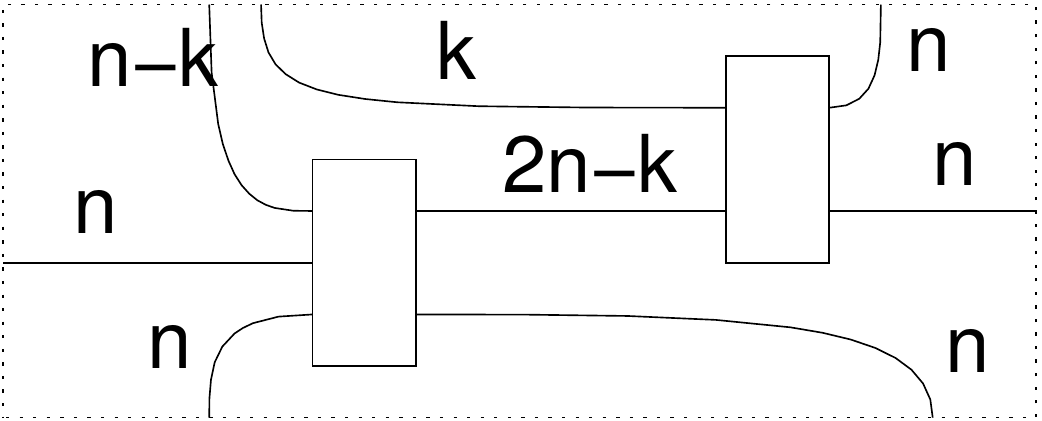} \raisebox{15pt}{$\; = \;$} \includegraphics[height=.5in]{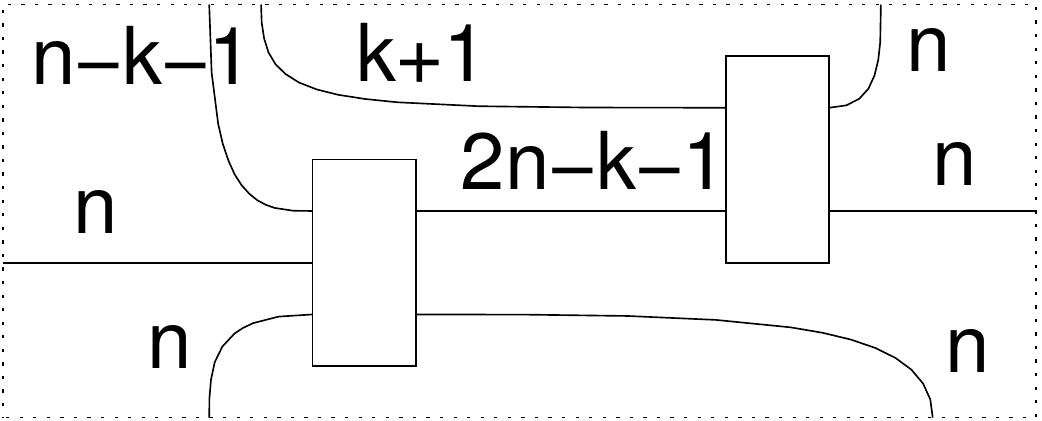} \raisebox{15pt}{$\; + \; (-1)^{k} \left(\frac{\Delta_{n-1}}{\Delta_{3n-k-1}}\right)\;$} \includegraphics[height=.5in]{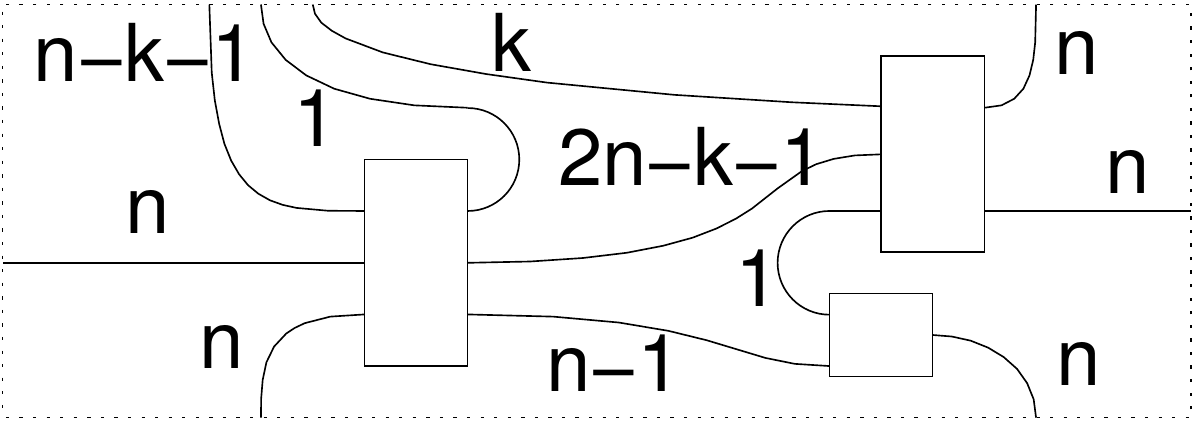}$$
Note that the two pictures on the left, call them $D_1$ and $D_2$, are both adequate diagrams and that their minimum degrees are the same by Lemma~\ref{replace}. Now we need to compare this with the minimum degree of the right-most term, call it $D_3$. First note the number of circles in $\bar{D_3}$ are $2n-k$ fewer than the number of circles in $\bar{D_2}$. This is because the diagrams $\bar{D_2}$ and $\bar{D_3}$ differ in only one spot where strands running straight across in $\bar{D_2}$ are replaced by a diagram as in Figure~\ref{merged}. 
\begin{figure}[htbp] %
   \centering
   $\includegraphics[height=1in, angle=90]{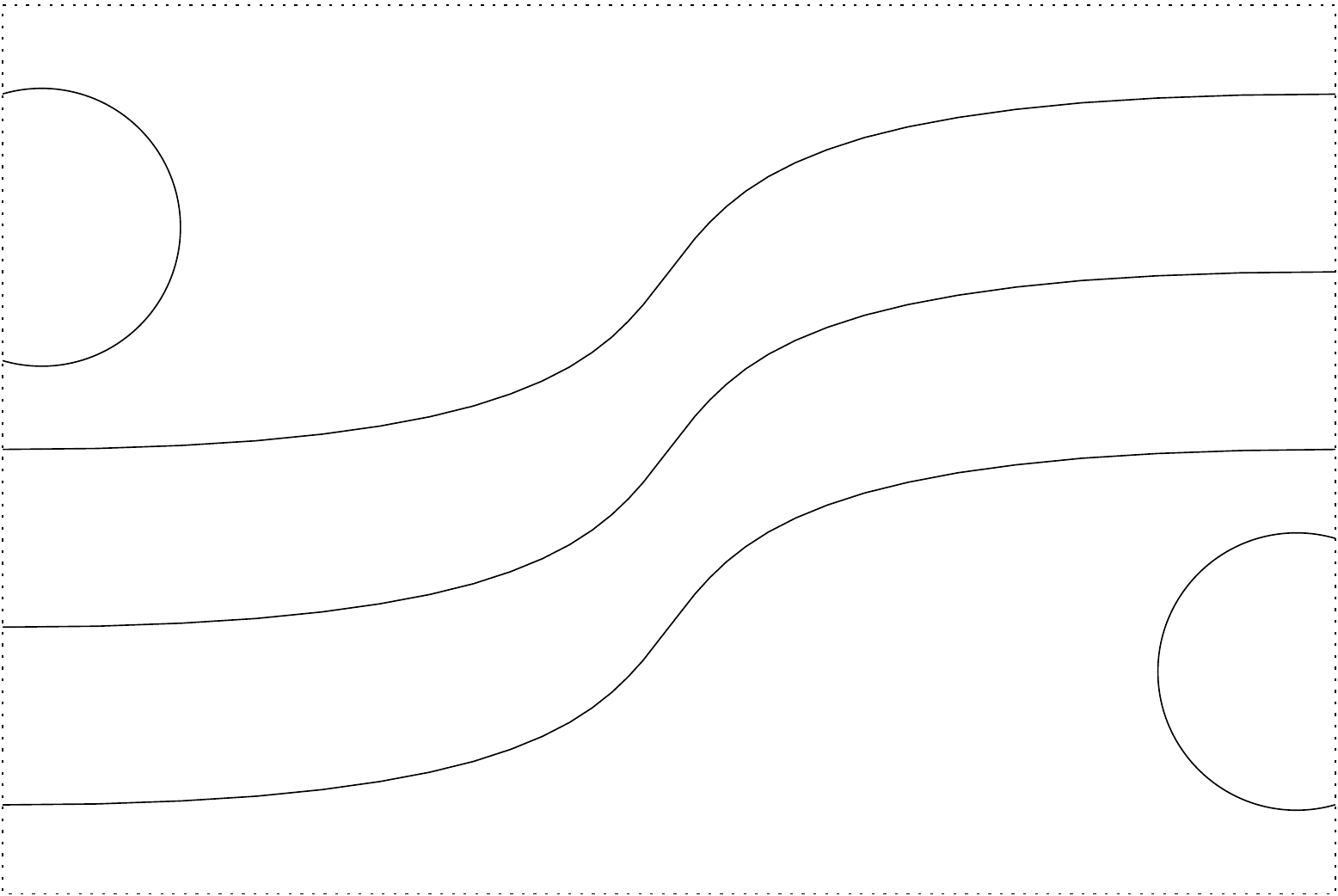} \raisebox{20pt}{$\;= \;$}\includegraphics[height=1in, angle=90]{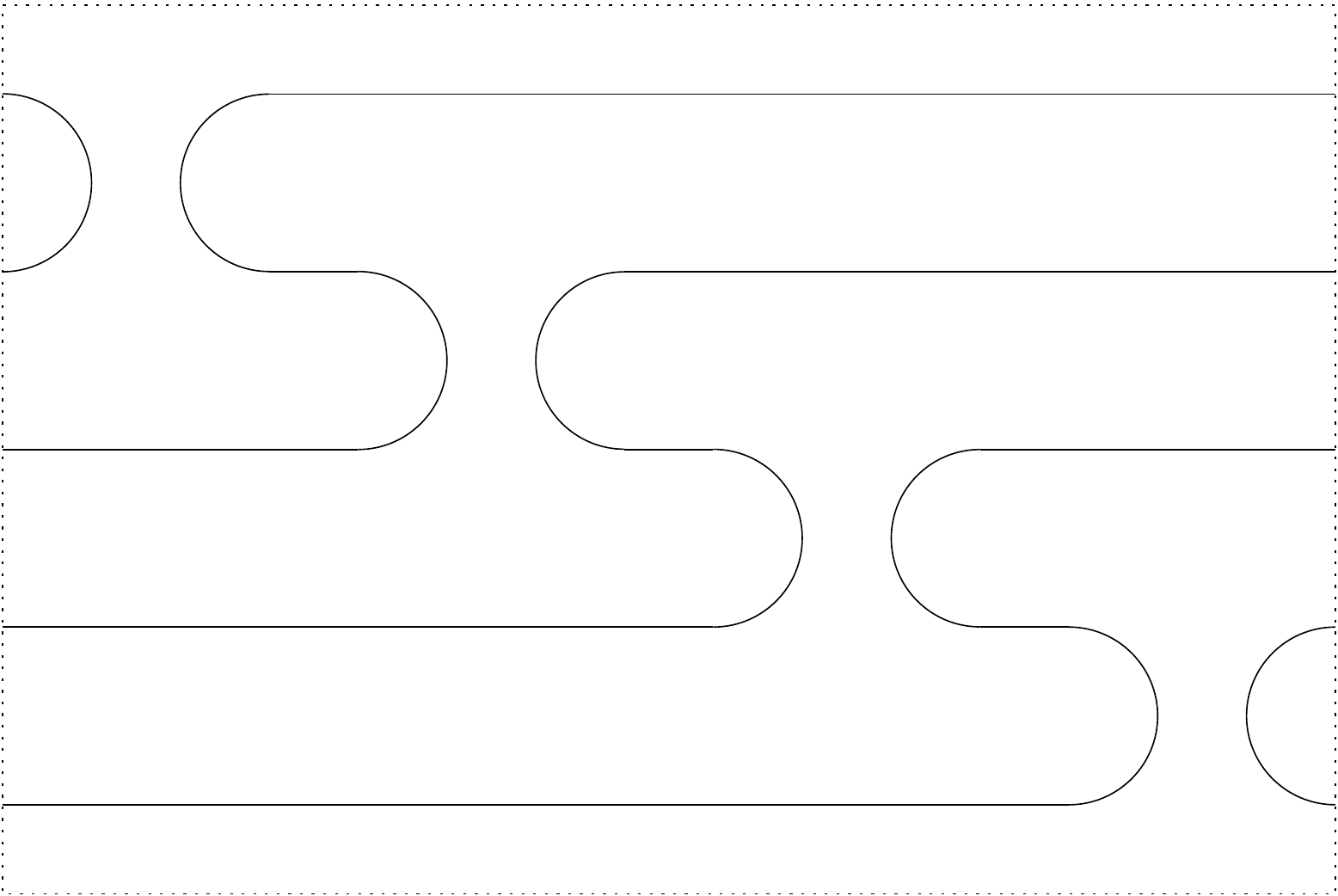}$
   \caption{$2n-k+1$ circles merged into $1$}
   \label{merged}
\end{figure}
Because $D_2$ is adequate, the different strands in Figure~\ref{merged} are closed to form $2n-k+1$ different circles. However, the corresponding strands in $D_3$ are merged in such a way that they close to form a single circle. Thus
$$d \left(\raisebox{-15pt}{\includegraphics[height=.5in]{SkeinPfP3.pdf}}\right) \; \geq \; d \left(\raisebox{-15pt}{\includegraphics[height=.5in]{SkeinPfP2.pdf}}\right) \; + 4n - 2k$$
Also note that 
$$d\left(\frac{\Delta_{n-1}}{\Delta_{3n-k-1}}\right) = 4n-2k$$
Finally, note that because $k<n$, we have $4(2n-k)\geq 4(n+1)$. Therefore:
$$\includegraphics[height=.5in]{SkeinPfP1.pdf} \raisebox{15pt}{$\; \dot{=}_{4(n+1)} \;$} \includegraphics[height=.5in]{SkeinPfP2.pdf}$$

Now by induction on $k$, this completes the proof.\end{proof}

The final step to prove Main Theorem~\ref{Main Theorem} is to verify the claim mentioned at the beginning of the section that if two $A$-adequate link diagrams $D_1$ and $D_2$ have all-$A$ smoothings that differ as in Main Theorem~\ref{Main Theorem}, then $S^{(n)}_{B_1} \;\dot{=}_{4(n+1)}\; S^{(n)}_{B_2}$. This follows from Lemma~\ref{big}. Denote the figure on the right of the equation in Lemma~\ref{big} $S^{(n)}_{B_2}$ because it matches the diagram on the right of the equation in Main Theorem~\ref{Main Theorem}. To get an expression for $S^{(n)}_{B_1}$ we can reflect all diagrams in Lemma~\ref{big} horizontally. Because the diagram on the left of the equation is symmetric, this shows that $S^{(n)}_{B_1} \;\dot{=}_{4(n+1)}\; S^{(n)}_{B_2}$, and Main Theorem~\ref{Main Theorem} follows.

\bibliography{AdequateHeads}

\providecommand{\bysame}{\leavevmode\hbox to3em{\hrulefill}\thinspace}
\providecommand{\MR}{\relax\ifhmode\unskip\space\fi MR }
\providecommand{\MRhref}[2]{%
  \href{http://www.ams.org/mathscinet-getitem?mr=#1}{#2}
}
\providecommand{\href}[2]{#2}
\begin{thebibliography}{FKP13}

\bibitem[AD11]{ArmondDasbach:RogersRamanujan}
Cody Armond and Oliver~T. Dasbach, \emph{{Rogers-Ramanujan type identities and
  the head and tail of the colored Jones polynomial}}, arXiv:1106.3948 (2011),
  1--27.

\bibitem[Arm11]{Armond:Walks}
Cody Armond, \emph{{Walks along braids and the colored Jones polynomial}},
  arXiv:1101.3810 (2011), 1--26.

\bibitem[Arm13]{Armond:HeadAndTailConjecture}
\bysame, \emph{{The head and tail conjecture for alternating knots}}, Alg.
  Geom. Top. \textbf{13} (2013), 2809--2826.

\bibitem[BN11]{BarNatan:KnotTheory}
Dror Bar-Natan, \emph{{KnotTheory, http://katlas.org}}, 2011.

\bibitem[DL06]{DasbachLin:HeadAndTail}
Oliver~T. Dasbach and Xiao-Song Lin, \emph{{On the head and the tail of the
  colored Jones polynomial}}, Compositio Mathematica \textbf{142} (2006),
  no.~05, 1332--1342.

\bibitem[DL07]{DL:VolumeIsh}
\bysame, \emph{{A volumish theorem for the Jones polynomial of alternating
  knots}}, Pacific J. Math. \textbf{231} (2007), no.~2, 279--291.

\bibitem[DT13]{DT:RefinedBound}
Oliver Dasbach and Anastasiia Tsvietkova, \emph{{A refined upper bound for the
  hyperbolic volume of alternating links and the colored Jones polynomial}},
  arXiv preprint math.GT/1310.0788 (2013), 10.

\bibitem[FKP08]{FKP:VolumeJones}
David Futer, Efstratia Kalfagianni, and Jessica~S. Purcell, \emph{{Dehn
  filling, volume, and the Jones polynomial}}, J. Differential Geom.
  \textbf{78} (2008), no.~3, 429--464.

\bibitem[FKP13]{FKP:Guts}
\bysame, \emph{{Guts of surfaces and the colored Jones polynomial}}, Lecture
  Notes in Mathematics, vol. 2069, Springer, Heidelberg, 2013.

\bibitem[Fut13]{Futer:Fiberedness}
David Futer, \emph{{Fiber detection for state surfaces}}, Alg. Geom. Top.
  \textbf{13} (2013), no.~5, 2799--2807.

\bibitem[GL11]{GL:NahmSums}
Stavros Garoufalidis and Thang T.~Q. L\^{e}, \emph{{Nahm sums, stability and
  the colored Jones polynomial}}, arXiv:1112.3905 (2011).

\bibitem[Lic97]{Lickorish:KnotTheoryBook}
W.~B.~Raymond Lickorish, \emph{{An introduction to knot theory}}, Springer,
  1997.

\bibitem[MV94]{MasbaumVogel:3valentGraphs}
Gregor Masbaum and Pierre Vogel, \emph{{3-valent graphs and the Kauffman
  bracket}}, Pacific J. Math. \textbf{164} (1994), no.~2, 361--381.

\bibitem[Roz12]{Rozansky:HeadAndTail}
Lev Rozansky, \emph{{Khovanov homology of a unicolored B-adequate link has a
  tail}}, arXiv:1203.5741 (2012).

\bibitem[Sto06]{StT:MutationCJP}
Alexander Stoimenow, \emph{{Mutation and the colored Jones polynomial}}, J.
  G\"{o}kova Geom. Top. \textbf{3} (2006), 44--78.

\end{thebibliography}
\bibliographystyle {amsalpha}

\end{document}